\documentclass{amsart}
\usepackage{amsmath}
\usepackage{amsfonts}

\newtheorem{theorem}{Theorem}
\theoremstyle{plain}

\newtheorem{definition}{Definition}
\newtheorem{example}{Example}

\newtheorem{lemma}{Lemma}

\newtheorem{remark}{Remark}

\numberwithin{equation}{section}

\begin{document}
\title[Relation between some types of convex functions]{On relations between
some types of convex functions }
\author{Shoshana Abramovich}
\address{Department of Mathematics, University of Haifa, Haifa, ISRAEL}
\email{abramos@math.haifa.ac.il}
\date{August 12, 2024}
\subjclass{26A15, }
\keywords{$\phi $-convexity, Superquadracity, Strong-convexity, Uniform
convexity, Error function }

\begin{abstract}
In this paper we show how the superquadratic functions can be used as a tool
for researching other types of convex functions like $\phi $-convexity,
strong-convexity and uniform convexity. We show how to use inequalities
satisfied by superquadratic functions and how to adapt the technique used to
get them in order to obtain new results satisfied by uniformly convex
functions and to $\phi $-convex functions.

Also, we show examples that emphasize relations between superquadracity and
some other types of convex functions
\end{abstract}

\maketitle

\section{\textbf{Introduction}}

In this paper we show how the superquadratic functions can be used as a tool
for researching other types of convex functions like $\phi $-convexity,
strong-convexity and uniform convexity. We show how to use inequalities
satisfied by superquadratic functions and how to adapt the technique used to
get them in order to obtain new results satisfied by uniformly convex
functions and to $\phi $-convex functions.

Also, we show examples that emphasize relations between superquadracity and
some other types of convex functions

We start with definitions of these classes of functions:

\begin{definition}
\label{Def1} \cite[Definition 1]{A2} A function $\varphi :\left[ 0,\infty
\right) \rightarrow 
\mathbb{R}
$ is superquadratic provided that for all $x\geq 0$\ there exists a constant 
$C_{x}\in $\ $%
\mathbb{R}
$ such that%
\begin{equation}
\varphi \left( y\right) \geq \varphi \left( x\right) +C_{x}\left( y-x\right)
+\varphi \left( \left\vert y-x\right\vert \right)  \label{1.1}
\end{equation}%
for all $y\geq 0$.\ If the reverse of (\ref{1.1}) holds then $\varphi $\ is
called subquadratic.
\end{definition}

\begin{lemma}
\label{Lem1} \cite[Lemma1]{A2} Let the function $\varphi $\bigskip\ be
superquadratic with $C_{x}$ as in Definition \ref{Def1}.

(i)\qquad Then $\bigskip \varphi \left( 0\right) \leq 0$

(ii)\qquad If $\varphi \left( 0\right) =\varphi ^{\prime }\left( 0\right) =0$%
, then $C_{x}=\varphi ^{\prime }\left( x\right) $ whenever~$\varphi $ is
differentiable at $x>0\bigskip .\ $

(iii)\qquad If $\varphi \geq 0,$ then $\varphi $\ is convex and \ $\varphi
\left( 0\right) =\varphi ^{\prime }\left( 0\right) =0.$
\end{lemma}

\begin{theorem}
\bigskip\ \label{Th1} \cite{AJS} The inequality 
\begin{equation}
\varphi \left( \int fd\mu \right) \leq \int \left( \varphi \left( f\left(
s\right) \right) -\varphi \left( \left\vert f\left( s\right) -\int fd\mu
\right\vert \right) \right) d\mu \left( s\right)  \label{1.2}
\end{equation}%
holds for all probability measures $\mu $ and all non-negative, $\mu $%
-integrable functions $f$ if and only if $\varphi $ is superquadratic.The
discrete version of this inequality is%
\begin{equation*}
\varphi \left( \sum_{r=1}^{n}\lambda _{r}x_{r}\right) \leq
\sum_{r=1}^{n}\lambda _{r}\left( \varphi \left( x_{r}\right) -\varphi \left(
\left\vert x_{r}-\sum_{j=1}^{n}\lambda _{j}x_{j}\right\vert \right) \right)
\end{equation*}%
for $x_{r}\geq 0,$ $\lambda _{r}\geq 0,$ \ $r=1,...,n$ and $%
\sum_{r=1}^{n}\lambda _{r}=1.$
\end{theorem}

\begin{remark}
\label{Rem1} The functions $\varphi \left( x\right) =x^{p},$ $x\geq 0$, are
superquadratic for $p\geq 2$ and subquadratic when $0\leq p\leq 2$.
\end{remark}

\begin{lemma}
\label{Lem2} \cite[Lemma 4.1]{AJS} A non-positive, non-increasing,
superadditive function is superquadratic.
\end{lemma}

\begin{definition}
\label{Def2} \cite{A} Let $I=\left[ a,b\right] \subset 
\mathbb{R}
$ be an interval and $\Phi :\left[ 0,b-a\right] \rightarrow 
\mathbb{R}
$ be a function. A function $f:\left[ a,b\right] \rightarrow 
\mathbb{R}
$ is said\ to be \textbf{generalized }$\Phi $\textbf{-uniformly convex}\ if:%
\begin{eqnarray*}
tf\left( x\right) +\left( 1-t\right) f\left( y\right) &\geq &f\left(
tx+\left( 1-t\right) y\right) +t\left( 1-t\right) \Phi \left( \left\vert
x-y\right\vert \right) \\
\text{for \ }x,y &\in &I\text{ \ and }t\in \left[ 0,1\right] \text{.}
\end{eqnarray*}%
If in addition $\Phi \geq 0$, then $f$\ is said to be $\Phi $\textbf{%
-uniformly convex}\textit{, }or\textit{\ }\textbf{uniformly convex with
modulus }$\Phi $.
\end{definition}

\begin{remark}
\label{Rem2} The functions $f\left( x\right) =x^{n},$ $x\geq 0,$ $n=2,3...$,
are uniformly convex on $x\geq 0$ with a modulus $\Phi \left( x\right)
=x^{n},$ $x>0$ (see \cite{A1}), (these functions are also superquadratic).
\end{remark}

\begin{remark}
\label{Rem3} It is proved in \cite{Z} that when $f$ is uniformly convex,%
\textbf{\ }there is always a modulus $\Phi $ which is increasing and $\Phi
\left( 0\right) =0$. It is also shown in \cite{Z} that the inequality%
\begin{equation}
f\left( \sum_{r=1}^{n}\lambda _{r}x_{r}\right) \leq \sum_{r=1}^{n}\lambda
_{r}\left( f\left( x_{r}\right) -\Phi \left( \left\vert
x_{r}-\sum_{j=1}^{n}\lambda _{j}x_{j}\right\vert \right) \right)  \label{1.3}
\end{equation}%
holds, and the inequality 
\begin{equation}
f\left( \int gd\mu \right) \leq \int \left( f\left( g\left( s\right) \right)
-\Phi \left( \left\vert g\left( s\right) -\int gd\mu \right\vert \right)
\right) d\mu \left( s\right)  \label{1.4}
\end{equation}%
also holds for all probability measures $\mu $ and all non-negative $\mu $%
-integrable functions $g$.
\end{remark}

\begin{definition}
\label{Def3} \cite{GP} A real value function $f$ defined on a real interval $%
I$ is called $\phi $-convex if for all $x,y\in I$, $t\in \left[ 0,1\right] $
it satisfies%
\begin{equation*}
tf\left( x\right) +\left( 1-t\right) f\left( y\right) +t\phi \left( \left(
1-t\right) \left\vert x-y\right\vert \right) +\left( 1-t\right) \phi \left(
t\left\vert x-y\right\vert \right) \geq f\left( tx+\left( 1-t\right)
y\right) ,
\end{equation*}%
where $\phi :\left[ 0,l\left( I\right) \right] \rightarrow 
\mathbb{R}
_{+},$ ($l\left( I\right) $\ is the length of the interval $I)$,\ is a
non-negative error function. In case of a $\phi $-convex function where $%
\phi $ satisfies the $\Gamma $ property, there is a $\phi $ which is called 
\textbf{optimal error function}.
\end{definition}

\begin{remark}
\label{Rem4} In \cite[Theorem 3]{GP} it is proved that $f$ is $\phi $-convex
iff the inequality 
\begin{equation*}
f\left( \sum_{r=1}^{n}\lambda _{r}x_{r}\right) \leq \sum_{r=1}^{n}\lambda
_{r}\left( f\left( x_{r}\right) +\phi \left( \left\vert
x_{r}-\sum_{j=1}^{n}\lambda _{j}x_{j}\right\vert \right) \right)
\end{equation*}%
holds, from which it follows that 
\begin{equation*}
f\left( \int gd\mu \right) \leq \int \left( f\left( g\left( s\right) \right)
+\phi \left( \left\vert g\left( s\right) -\int gd\mu \right\vert \right)
\right) d\mu \left( s\right)
\end{equation*}%
also holds for all probability measures $\mu $ and all non-negative $\mu $%
-integrable functions $g$.
\end{remark}

\begin{remark}
\label{Rem5} The functions $f\left( x\right) =-x^{p},$ $x\geq 0,$ $1\leq
p\leq 2$ are $\phi $-convex with error function $\phi =-f$ (see \cite{AJS})
because $\left( -f\right) $ is superquadratic and negative as all
superquadratic functions which are negative on an interval $\left[ 0,A\right]
$, $A>0$, are $\phi $-convex.
\end{remark}

\begin{definition}
\label{Def4} \cite{GP} We say tht an error function $\phi \in E\left(
I\right) =\left[ 0,l\left( y\right) \right] $ possesses the property $\Gamma 
$ if it satisfies the inequality:%
\begin{equation}
\phi \left( x+y\right) \leq \phi \left( x\right) +\phi \left( y\right) +2%
\frac{x}{y}\phi \left( y\right) ,\quad x\geq 0,\quad y>0,\quad \left(
x+y\right) <l\left( I\right) .  \label{1.5}
\end{equation}%
The class of error function in $E\left( I\right) $\ with the property $%
\Gamma $ is denoted by $E^{\Gamma }\left( I\right) $. The subset of $%
E^{\Gamma }\left( I\right) $ whose elements also satisfy $\phi \left(
0\right) =0$ is denoted by $E_{0}^{\Gamma }\left( I\right) $.
\end{definition}

The reverse of Inequality (\ref{1.5}) is satisfied by all superquadratic
functions $f$ as proved in \cite{KMS}, that is, all the superquadratic
functions posess the property $\left( -\Gamma \right) $.

However, only few of $\phi $-convex functions posess the propety $\Gamma $
as shown in \cite{GP}, and these specific $\phi $-convex functions lead to
important inequalities:

\begin{theorem}
\label{Th2}\textbf{\ }\cite[Theorem 1]{KMS}.\ Let $f$ be a superquadratic
function. Then, when $a,b\geq 0$ the inequalities%
\begin{equation*}
f\left( a\right) +f\left( b\right) \leq f\left( a+b\right) -\frac{2a}{a+b}%
f\left( b\right) -\frac{2b}{a+b}f\left( a\right) ,\quad a+b>0
\end{equation*}%
and 
\begin{equation*}
f\left( a+b\right) \geq f\left( a\right) +f\left( b\right) +2\frac{a}{b}%
f\left( b\right) ,\quad a\geq 0,\quad b>0.
\end{equation*}%
hold for $a\geq 0,$ $b>0$.
\end{theorem}

\begin{theorem}
\label{Th3} \cite[Theorem 4.1]{GP}.\textit{\ Let }$\phi \in E^{\Gamma
}\left( I\right) $\textit{, that is }$\phi $\textit{\ posesses the property }%
$\Gamma $\textit{: Then }$\sqrt{\phi }$\textit{\ and }$\frac{\phi \left(
t\right) }{t}$\textit{\ are subadditive on }$\left[ 0,l\left( I\right) %
\right] $\textit{. If in addition }$\varphi :\left[ 0,l\left( I\right) %
\right] \rightarrow 
\mathbb{R}
_{+}$\textit{\ is decreasing on }$\left[ 0,l\left( I\right) \right] $\textit{%
, then }$\Psi =\varphi \phi \in E^{\Gamma }\left( I\right) $\textit{. In
particular }$\Psi \in E\left( I\right) $\textit{\ and }$\frac{\Psi \left(
t\right) }{t^{2}}$\textit{\ is decreasing on }$\left[ 0,l\left( I\right) %
\right] $\textit{, then }$\Psi \in $\textit{\ }$E^{\Gamma }\left( I\right) $%
\textit{.}
\end{theorem}

The following theorem is an external version of Jensen inequality for
superquadratic functions:

\begin{theorem}
\bigskip \label{Th4} \cite[Theorem 1]{BPV} Let $f$ be a superquadratic
function, let $x_{i}\geq 0,$ $i=1,...,n$, $\nu _{n}\geq 1,$ $\nu _{i}\leq 0,$
$i=1,...,n-1$ and let $\sum_{i=1}^{n}\nu _{i}x_{i}>0$ then the inequality 
\begin{eqnarray}
&&f\left( \sum_{i=1}^{n}\nu _{i}x_{i}\right)  \label{1.6} \\
&\geq &\sum_{i=1}^{n}\nu _{i}f\left( x_{i}\right) +f\left( \left\vert
\sum_{i=1}^{n}\nu _{i}x_{i}-x_{n}\right\vert \right) -\sum_{i=1}^{n-1}\nu
_{i}f\left( \left\vert x_{i}-x_{n}\right\vert \right)  \notag
\end{eqnarray}%
holds.
\end{theorem}

Theorem \ref{Th4} was proved in 2008 in \cite[Theorem 1]{BPV}. It was proved
again only for the special case $n=2$ in \cite[Theorem 2]{KMS} in 2014. It
says in this case:

\begin{theorem}
\label{Th5} \cite[Theorem 2]{KMS}, \cite[Theorem 3, for $n=2$]{BPV} . Let $%
g\ $ be a superquadratic function and $a,b\geq 0$. If\ $\nu <0$ and $\left(
1-\nu \right) a+\nu b\geq 0$, then the inequality 
\begin{eqnarray}
&&\left( 1-\nu \right) g\left( a\right) +\nu g\left( b\right)  \label{1.7} \\
&\leq &g\left( \left( 1-\nu \right) a+\nu b\right) +\nu g\left( \left\vert
a-b\right\vert \right) -g\left( \nu \left\vert a-b\right\vert \right)  \notag
\end{eqnarray}%
holds, while for $\nu >1$, then the inequality 
\begin{eqnarray}
&&\left( 1-\nu \right) g\left( a\right) +\nu g\left( b\right)  \label{1.8} \\
&\leq &g\left( \left( 1-\nu \right) a+\nu b\right) +\left( 1-\nu \right)
g\left( \left\vert a-b\right\vert \right) -g\left( \left( \nu -1\right)
\left\vert a-b\right\vert \right) ,  \notag
\end{eqnarray}%
holds, provided that \ $\left( 1-\nu \right) a+\nu b\geq 0$.
\end{theorem}

A refinement of Hermite-Hadamard inequality for superquadratic functions is
the following Theorem \ref{Th6}.

\begin{theorem}
\label{Th6} \cite[Theorem 8]{BPV} Let $\varphi :\left[ 0,\infty \right)
\rightarrow 
\mathbb{R}
$ be an integrable superquadratic function, $0\leq a<b$. Then%
\begin{eqnarray*}
&&\varphi \left( \frac{a+b}{2}\right) +\frac{1}{b-a}\int_{a}^{b}\varphi
\left( \left\vert t-\frac{a+b}{2}\right\vert \right) dt \\
&\leq &\frac{1}{b-a}\int_{a}^{b}\varphi \left( t\right) dt \\
&\leq &\frac{\varphi \left( a\right) +\varphi \left( b\right) }{2} \\
&&-\frac{1}{\left( b-a\right) ^{2}}\int_{a}^{b}\left[ \left( b-t\right)
\varphi \left( t-a\right) +\left( t-a\right) \varphi \left( b-t\right) %
\right] dt.
\end{eqnarray*}
\end{theorem}

A. Cipu \cite{CIPU} proved the following theorem:

\begin{theorem}
\label{Th7} Let $n>1$\ be an integer and $x_{1},x_{2},...,x_{n}$\ be
positive real numbers. Denote $a=\frac{1}{n}\sum_{i=1}^{n}x_{i}$\ \ and \ $b=%
\frac{1}{n}\sum_{i=1}^{n}x_{i}^{2},$\ \ then $\ \underset{1\leq k\leq n}{%
\max }\left\{ \left\vert x_{k}-a\right\vert \right\} \leq \sqrt{\left(
n-1\right) \left( b-a^{2}\right) }.$
\end{theorem}

Theorems \ref{Th8} and \ref{Th9} are stated and proved in \cite{ABP}:

\begin{theorem}
\label{Th8} Let $n>1$ be an integer and $x_{1},x_{2},...,x_{n}$ be positive
real numbers. Denote \ $a=\sum_{i=1}^{n}\alpha _{i}x_{i}$ \ and \ $%
c=\sum_{i=1}^{n}\alpha _{i}x_{i}^{p},$ \ where \ $0<\alpha _{i}<1,$ $\
i=1,...,n\ \ \ \sum \alpha _{i}=1,$ \ $p\geq 2,$ then \ \ 
\begin{equation}
\ \underset{1\leq k\leq n}{\max }\left\{ \left\vert x_{k}-a\right\vert
\right\} \leq T\left( c-a^{p}\right) ^{\frac{1}{p}}  \label{1.9}
\end{equation}%
where 
\begin{equation*}
T=\frac{\left( 1-\alpha _{0}\right) ^{1-\frac{1}{p}}}{\alpha _{0}^{\frac{1}{p%
}}\left( \alpha _{0}^{p-1}+\left( 1-\alpha _{0}\right) ^{p-1}\right) ^{\frac{%
1}{p}}},\qquad \alpha _{0}=\underset{1\leq k\leq n}{\min }\left( \alpha
_{k}\right) .
\end{equation*}
\end{theorem}

\begin{remark}
\label{Rem6}It is easy to see that Theorem \ref{Th7} is a special case of
Theorem \ref{Th8} where \ $p=2,$ $\ \alpha _{k}=\frac{1}{n},$ \ $k=1,...,n$.
\end{remark}

\begin{theorem}
\label{Th9} Let $f$ be a positive superquadratic function on $\left[
0,\infty \right) $, which satisfies \ $f\left( AB\right) \leq f\left(
A\right) f\left( B\right) $ \ for $A>0$, $B>0$.

Let $x_{1},x_{2},...,x_{n}$ be positive real numbers. Denote $%
a=\sum_{i=1}^{n}\frac{x_{i}}{n}$ \ and \ $d=\frac{1}{n}\sum_{i=1}^{n}f\left(
x_{i}\right) ,$ \ then 
\begin{equation}
\underset{1\leq k\leq n}{\max }\left( f\left( \left\vert x_{k}-a\right\vert
\right) \right) \leq \frac{f\left( n-1\right) n}{n-1+f\left( n-1\right) }%
\left( d-f\left( a\right) \right) .  \label{1.10}
\end{equation}
\end{theorem}

\begin{remark}
\label{Rem7} For \ $\alpha _{i}=\frac{1}{n},$\ \ $i=1,...,n$\ \ Theorem \ref%
{Th8} is a special case of Theorem \ref{Th9}.
\end{remark}

In Section 2 we use theorems \ref{Th4}, \ref{Th5} and \ref{Th6} which are
related to superquadratic functions to get similar results related to $\phi $%
-convex functions and uniformly convex functions.

In Section 3 we use theorems \ref{Th7}, \ref{Th8} and \ref{Th9} which are
related to superquadratic functions to get similar results related to $\phi $%
-convex functions and uniformly convex functions.

In Section 4 we show examples about relations between superquadracity and
some other types of convex functions.

The results and the examples shown in this paper are just few cases out of
many others, where we use similar techniques of these generating results for
superquadratic functions to get new results related to uniform convexity and 
$\phi $-convexity. For instance we can adapt the inequalities obtained by
using Theorem \ref{Th2} proved in \cite{KMS} only to the special cases of $%
\phi $-convex functions with error function $\phi $ which satisfy Property $%
\Gamma $, among them the optimal error function (see Theorem \ref{Th3}).

\section{\textbf{External versions of Jensen inequality and Hermite-Hadamard
inequality for }$\protect\phi $\textbf{-convex functions and uniform convex
functions.}}

Theorem \ref{Th10} deals with $\phi $-convex functions and uniformly convex
functions. The proof of the Theorem \ref{Th10} is similar to the proof of
Theorem \ref{Th4} given in \cite[Theorem 1]{BPV} for the convenience of the
reader we provide here a complete proof of Theorem \ref{Th10} for $\phi $%
-convex functions and uniformly convex functions.

\begin{theorem}
\label{Th10} Let $f:\left[ a,b\right] \rightarrow 
\mathbb{R}
$ where $\phi :\left[ 0,b-a\right] \rightarrow 
\mathbb{R}
_{+}$. \ Let $x_{i}\in \left( a,b\right) ,$ $i=1,...,n$, $\nu _{n}\geq 1,$ $%
\nu _{i}\leq 0,$ $i=1,...,n-1$ and let $\sum_{i=1}^{n}\nu _{i}x_{i}\in
\left( a,b\right) $ then:

a. \ \ if $f$ is $\phi $-convex, where the error function is $\phi $, the
inequality 
\begin{eqnarray}
&&f\left( \sum_{i=1}^{n}\nu _{i}x_{i}\right)  \label{2.1} \\
&\geq &\sum_{i=1}^{n}\nu _{i}f\left( x_{i}\right) -\phi \left( \left\vert
\sum_{i=1}^{n}\nu _{i}x_{i}-x_{n}\right\vert \right) +\sum_{i=1}^{n-1}\nu
_{i}\phi \left( \left\vert x_{i}-x_{n}\right\vert \right)  \notag
\end{eqnarray}%
holds.

In the special case that $n=2$ we get that%
\begin{eqnarray}
\left( 1-\nu \right) f\left( a\right) +\nu f\left( b\right) &\leq &f\left(
\left( 1-\nu \right) a+\nu b\right) -\left( 1-\nu \right) \phi \left(
\left\vert a-b\right\vert \right)  \label{2.2} \\
&&+\phi \left( \left( \nu -1\right) \left\vert a-b\right\vert \right) , 
\notag
\end{eqnarray}%
holds.

b. \ \ If $f$ is uniformly convex with modulus $\Phi $,\ the inequality 
\begin{eqnarray}
&&f\left( \sum_{i=1}^{n}\nu _{i}x_{i}\right)  \label{2.3} \\
&\geq &\sum_{i=1}^{n}\nu _{i}f\left( x_{i}\right) +\Phi \left( \left\vert
\sum_{i=1}^{n}\nu _{i}x_{i}-x_{n}\right\vert \right) -\sum_{i=1}^{n-1}\nu
_{i}\Phi \left( \left\vert x_{i}-x_{n}\right\vert \right)  \notag
\end{eqnarray}%
holds.

In the special case that $n=2$ we get that%
\begin{equation}
f\left( \nu _{1}x_{1}+\nu _{2}x_{2}\right) \geq \nu _{1}f\left( x_{1}\right)
+\nu _{2}f\left( x_{2}\right) -\frac{\nu _{1}}{\nu _{2}}\Phi \left( \nu
_{2}\left( \left\vert x_{2}-x_{1}\right\vert \right) \right)  \label{2.4}
\end{equation}%
holds.
\end{theorem}

\begin{proof}
a. \ \ From equality $f\left( x_{n}\right) =f\left( \frac{1}{\nu _{n}}\left(
\sum_{i=1}^{n}\nu _{i}x_{i}\right) +\sum_{i=1}^{n-1}\frac{-\nu _{i}}{\nu _{n}%
}x_{i}\right) $ as $\frac{1}{\nu _{n}}+\sum_{i=1}^{n-1}\frac{-\nu _{i}}{\nu
_{n}}=1$\ and from the $\phi $-convexity of $f$\ we get that 
\begin{eqnarray}
f\left( x_{n}\right) &=&f\left( \frac{1}{\nu _{n}}\left( \sum_{i=1}^{n}\nu
_{i}x_{i}\right) +\sum_{i=1}^{n-1}\frac{-\nu _{i}}{\nu _{n}}x_{i}\right)
\label{2.5} \\
&\leq &\frac{1}{\nu _{n}}f\left( \sum_{i=1}^{n}\nu _{i}x_{i}\right)
+\sum_{i=1}^{n-1}\frac{-\nu _{i}}{\nu _{n}}f\left( x_{i}\right)  \notag \\
&&-\frac{1}{\nu _{n}}\phi \left( \sum_{i=1}^{n}\left\vert \nu
_{i}x_{i}-x_{n}\right\vert \right) +\sum_{i=1}^{n-1}\frac{-\nu _{i}}{\nu _{n}%
}\phi \left( \left\vert x_{i}-x_{n}\right\vert \right) .  \notag
\end{eqnarray}%
From (\ref{2.5}) we get by multipying by $\nu _{n}$ that \ 
\begin{eqnarray}
\nu _{n}f\left( x_{n}\right) &\leq &f\left( \sum_{i=1}^{n}\nu
_{i}x_{i}\right) -\sum_{i=1}^{n-1}\nu _{i}f\left( x_{i}\right)  \label{2.6}
\\
&&+\phi \left( \sum_{i=1}^{n}\left\vert \nu _{i}x_{i}-x_{n}\right\vert
\right) -\sum_{i=1}^{n-1}\nu _{i}\phi \left( \left\vert
x_{i}-x_{n}\right\vert \right) ,  \notag
\end{eqnarray}%
which is the same as (\ref{2.1}). From (\ref{2.1}) when $n=2$ we get (\ref%
{2.2}).

b. \ \ The proof of (\ref{2.3}), for uniformly convex functions is similar
to the proof of (\ref{2.1}) for $\phi $-convex functions and therefore is
omitted.

To prove (\ref{2.4}) we use Definition \ref{Def2} of uniformly convex
functions with modulus $\Phi $\ and get when $\nu _{2}>1$ and $\nu _{1}<0$
that \ 
\begin{eqnarray}
f\left( x_{2}\right) &=&f\left( \frac{1}{\nu _{2}}\left( \nu _{1}x_{1}+\nu
_{2}x_{2}\right) +\left( \frac{-\nu _{1}}{\nu _{2}}\right) x_{1}\right)
\label{2.7} \\
&\leq &\left( \frac{1}{\nu _{2}}\right) f\left( \nu _{1}x_{1}+\nu
_{2}x_{2}\right) +\left( \frac{-\nu _{1}}{\nu _{2}}\right) f\left(
x_{1}\right)  \notag \\
&&+\frac{\nu _{1}}{\nu _{2}^{2}}\Phi \left( \nu _{2}\left( \left\vert
x_{2}-x_{1}\right\vert \right) \right) .  \notag
\end{eqnarray}%
Multipying this inequality by $\nu _{2}$ we get Inequality (\ref{2.4})

The proof is complete.
\end{proof}

\begin{example}
\label{Ex1} \ Let the function $f$ be \ 
\begin{equation*}
f\left( x\right) =2x^{n},\quad x\in \left[ a,a+1\right] ,\quad a\geq 1,\quad
n=2,3,...\text{.}
\end{equation*}%
Then $f$ is a uniformly convex function when its modulus $\Phi $ is: 
\begin{equation*}
\Phi =x^{n}\left( 3x-x^{3}\right) ,\quad x\in \left[ 0,1\right] \text{.}
\end{equation*}

Indeed, the function $g\left( x\right) =2x^{n}$ where $n=2,3,...$ is
uniformly convex (see \cite{A}), that is,%
\begin{eqnarray}
&&t2x^{n}+\left( 1-t\right) 2y^{n}-2\left( tx+\left( 1-t\right) y\right) ^{n}
\label{2.8} \\
&\geq &t\left( 1-t\right) 2\left\vert x-y\right\vert ^{n}.  \notag
\end{eqnarray}%
A simple computation shows that $x^{n}\left( 3x-x^{3}\right) \leq 2x^{n}$.
Hence, from (\ref{2.8}) we get that 
\begin{eqnarray*}
&&t2x^{n}+\left( 1-t\right) 2y^{n}-2\left( tx+\left( 1-t\right) y\right) ^{n}
\\
&\geq &t\left( 1-t\right) 2\left\vert x-y\right\vert ^{n} \\
&\geq &t\left( 1-t\right) \left\vert x-y\right\vert ^{n}\left( 3\left\vert
x-y\right\vert -\left\vert x-y\right\vert ^{3}\right) ,
\end{eqnarray*}%
that is, $f\left( x\right) =2x^{n}$ is a uniformly convex function with
modulus $\Phi \left( x\right) =x^{n}\left( 3x-x^{3}\right) $ as claimed.

As the function $f:\left[ a,a+1\right] \rightarrow 
\mathbb{R}
$\ is uniformly convex where $\Phi :\left[ 0,1\right] \rightarrow 
\mathbb{R}
_{+}$, it satisfies the inequality (\ref{2.7}) for $\nu >1$ and $\left(
1-\nu \right) a+\nu b\in \left[ a,a+1\right] ,$ $a\geq 1$,\ that is, the
Inequality (\ref{2.4}) holds.
\end{example}

The Hermite-Hadamard inequality versions for uniformly convex functions and
for $\phi $-convex functions are as follows:\ 

\begin{theorem}
\label{Th11} Let $f:\left[ a,b\right] \rightarrow 
\mathbb{R}
$ be an integrable $\phi $-convex function with error $\phi :\left[ 0,b-a%
\right] \rightarrow 
\mathbb{R}
_{+}$. Then%
\begin{eqnarray}
&&f\left( \frac{a+b}{2}\right) -\frac{1}{b-a}\int_{a}^{b}\phi \left(
\left\vert t-\frac{a+b}{2}\right\vert \right) dt  \label{2.9} \\
&\leq &\frac{1}{b-a}\int_{a}^{b}f\left( t\right) dt  \notag
\end{eqnarray}%
and%
\begin{equation}
\frac{1}{b-a}\int_{a}^{b}f\left( t\right) dt\leq \frac{f\left( a\right)
+f\left( b\right) }{2}\qquad \qquad \qquad \qquad \qquad  \label{2.10}
\end{equation}%
\begin{equation*}
+\frac{1}{\left( b-a\right) ^{2}}\int_{a}^{b}\left[ \left( b-t\right) \phi
\left( t-a\right) +\left( t-a\right) \phi \left( b-t\right) \right] dt.
\end{equation*}%
In case that $f:\left[ a,b\right] \rightarrow 
\mathbb{R}
$ is an integrable uniformly convex function with modulus $\Phi :\left[ 0,b-a%
\right] \rightarrow 
\mathbb{R}
_{+}$. Then%
\begin{eqnarray}
&&f\left( \frac{a+b}{2}\right) +\frac{1}{b-a}\int_{a}^{b}\Phi \left(
\left\vert t-\frac{a+b}{2}\right\vert \right) dt  \label{2.11} \\
&\leq &\frac{1}{b-a}\int_{a}^{b}f\left( t\right) dt  \notag
\end{eqnarray}%
and%
\begin{equation}
\frac{1}{b-a}\int_{a}^{b}f\left( t\right) dt\leq \frac{f\left( a\right)
+f\left( b\right) }{2}-\frac{1}{6}\Phi \left( \left\vert b-a\right\vert
\right)  \label{2.12}
\end{equation}
\end{theorem}

\begin{proof}
First we prove (\ref{2.9}) for $\phi $-convex functions. From the Corollary
of Definition \ref{Def3} for $f\left( t\right) =t$ and measure $\mu $\ on $%
\Omega =\left[ a,b\right] $, defined by $d\mu =\frac{1}{b-a}dt$, Inequality %
\ref{2.9} holds.

To get (\ref{2.10}) we use Definition \ref{Def3} for $t\rightarrow \frac{b-t%
}{b-a}$, $x=a,$ $y=b$, that is:%
\begin{equation*}
f\left( t\right) \leq \frac{b-t}{b-a}\left( f\left( a\right) +\phi \left(
t-a\right) \right) +\frac{t-a}{b-a}\left( f\left( b\right) +\phi \left(
b-t\right) \right)
\end{equation*}%
for all $t,$ such that $a<t<b$. After integrating this expression\ over the
segment $\left[ a,b\right] $:%
\begin{eqnarray*}
\int_{a}^{b}f\left( t\right) dt &\leq &\frac{1}{b-a}\int_{a}^{b}\left[
bf\left( a\right) -af\left( b\right) +t\left( f\left( b\right) -f\left(
a\right) \right) \right] dt \\
&&+\frac{1}{b-a}\int_{a}^{b}\left[ \left( b-t\right) \phi \left( t-a\right)
+\left( t-a\right) \phi \left( b-t\right) \right]
\end{eqnarray*}%
and dividing with $\left( b-a\right) $ we get (\ref{2.10}).

\ The proof Inequality (\ref{2.11}) for a uniformly convex function with a
modulus $\Phi $\ is similar to the proof of the Inequality (\ref{2.9}) for $%
\phi $-convex functions and therefore is omitted.

To prove (\ref{2.12}) we use Definition \ref{Def2} and get 
\begin{equation*}
f\left( t\right) \leq \frac{b-t}{b-a}f\left( a\right) +\frac{t-a}{b-a}%
f\left( b\right) -\left( \frac{b-t}{b-a}\right) \left( \frac{t-a}{b-a}%
\right) \Phi \left( \left\vert b-a\right\vert \right)
\end{equation*}%
for all $t,$ such that $a<t<b$. After integrating this expression\ over the
segment $\left[ a,b\right] $ we get Inequality (\ref{2.12}).

The proof is complete.
\end{proof}

\section{\textbf{Upper bounds for deviations from a Mean Value}}

\begin{theorem}
\label{Th12} Let $f$ be an uniformly convex function with modulus $\Phi
\left( x\right) =mx^{p},$ $p\geq 2,$ $m>0$,$\ $ that is, $f$ is a strongly
convex function. Let $n>1$ be an integer and $x_{1},x_{2},...,x_{n}$ be
positive real numbers. Denote \ $a=\sum_{i=1}^{n}\alpha _{i}x_{i}$ \ and \ $%
c=\sum_{i=1}^{n}\alpha _{i}f\left( x_{i}\right) ,$ \ where \ $0<\alpha
_{i}<1,$ $\ i=1,...,n\ \ \ \sum \alpha _{i}=1,$ \ $,$ then \ \ 
\begin{equation}
\ \underset{1\leq k\leq n}{\max }\left\{ \left\vert x_{k}-a\right\vert
\right\} \leq T\left( c-f\left( a\right) \right) ^{\frac{1}{p}}  \label{3.1}
\end{equation}%
where 
\begin{equation*}
T=\frac{\left( 1-\alpha _{0}\right) ^{1-\frac{1}{p}}}{\alpha _{0}^{\frac{1}{p%
}}\left( \alpha _{0}^{p-1}+\left( 1-\alpha _{0}\right) ^{p-1}\right) ^{\frac{%
1}{p}}},\qquad \alpha _{0}=\underset{1\leq k\leq n}{\min }\left( \alpha
_{k}\right) .
\end{equation*}
\end{theorem}

\begin{example}
\label{Ex2} Let $\varphi \left( x\right) ,$ $\geq 0$ be a convex function
satisfying $\varphi ^{^{\prime }}\left( 0\right) >0$. It is easy to verify
that the function $g\left( x\right) =x\varphi \left( x\right) $ satisfyes
the iequality 
\begin{equation*}
\sum_{i=1}^{n}\alpha _{i}g\left( x_{i}\right) -g\left( \sum_{j=1}^{n}\alpha
_{j}x_{j}\right) \geq \varphi ^{^{\prime }}\left( 0\right) \left(
\sum_{i=1}^{n}\alpha _{i}\left( x_{i}-\sum_{j=1}^{n}\alpha _{j}x_{j}\right)
^{2}\right) .
\end{equation*}%
Denoting $f\left( x\right) =\frac{g\left( x\right) }{\varphi ^{^{\prime
}}\left( 0\right) }$ we get that\ $f\left( x\right) $\ is strongly convex,
and therefore uniformly convex satisfy Theorem \ref{Th12}.

If \ $\varphi ^{^{\prime }}\left( 0\right) <0$ then $g$ is $\phi $-convex
when $\phi =-\varphi ^{^{\prime }}(0)x^{2}$. \ 
\end{example}

\begin{proof}
(of Theorem \ref{Th12}) \ $\Phi \left( x\right) =x^{p}$ \ $p\geq 2$,
therefore applying Remark \ref{Rem3}%
\begin{equation}
\sum_{i=1}^{n}\alpha _{i}f\left( x_{i}\right) -f\left( \sum_{k=1}^{n}\alpha
_{k}x_{k}\right) \geq \sum_{i=1}^{n}\alpha _{i}\left( \left\vert
x_{i}-a\right\vert \right) ^{p}  \label{3.2}
\end{equation}%
holds.

Denote $y_{i}=x_{i}-a,$ \ $i=1,...,n.$ \ Then from $\sum_{i=1}^{n}\alpha
_{i}y_{i}=0$ \ we get from H\"{o}lder's inequality 
\begin{eqnarray*}
\left( \alpha _{n}\left\vert y_{n}\right\vert \right) ^{p} &=&\left(
\left\vert \sum_{i=1}^{n-1}\alpha _{i}y_{i}\right\vert \right) ^{p}\leq
\left( \sum_{i=1}^{n-1}\alpha _{i}\left\vert y_{i}\right\vert \right) ^{p} \\
&=&\left( \sum_{i=1}^{n-1}\alpha _{i}^{1-\frac{1}{p}}\left( \alpha
_{i}\left\vert y_{i}\right\vert ^{p}\right) ^{\frac{1}{p}}\right) ^{p}\leq
\left( \sum_{i=1}^{n-1}\alpha _{i}\right) ^{p-1}\sum_{i=1}^{n-1}\alpha
_{i}\left\vert y_{i}\right\vert ^{p},
\end{eqnarray*}%
and 
\begin{equation*}
\left( \alpha _{n}\left\vert y_{n}\right\vert \right) ^{p}\leq \left(
1-\alpha _{n}\right) ^{p-1}\left( \sum_{i=1}^{n}\alpha _{i}\left\vert
y_{i}\right\vert ^{p}-\alpha _{n}\left\vert y_{n}\right\vert ^{p}\right) .
\end{equation*}%
Therefore from (\ref{3.2}) 
\begin{equation*}
\alpha _{n}\left( \alpha _{n}^{p-1}+\left( 1-\alpha _{n}\right)
^{p-1}\right) \left\vert y_{n}\right\vert ^{p}\leq \left( 1-\alpha
_{n}\right) ^{p-1}\left( c-f\left( a\right) \right)
\end{equation*}%
which by taking into consideration that $\frac{\left( 1-\alpha \right) ^{p-1}%
}{\alpha \left( \alpha ^{p-1}+\left( 1-\alpha \right) ^{p-1}\right) }$\ is
decreasing for $0<\alpha <1$ leads to (\ref{3.1}).
\end{proof}

\begin{theorem}
\label{Th13} Let $f$ $\in 
\mathbb{R}
$ be an uniformly convex function with modulus $\Phi $, which satisfies \ $%
\Phi \left( AB\right) \leq \Phi \left( A\right) \Phi \left( B\right) $ \ for 
$A>0$, $B>0$.

Let $x_{1},x_{2},...,x_{n}$ be positive real numbers. Denote $%
a=\sum_{i=1}^{n}\frac{x_{i}}{n}$ \ and \ $d=\frac{1}{n}\sum_{i=1}^{n}f\left(
x_{i}\right) $. If $\Phi $ is convex\ then 
\begin{equation}
\underset{1\leq k\leq n}{\max }\left( \Phi \left( \left\vert
x_{k}-a\right\vert \right) \right) \leq \frac{\Phi \left( n-1\right) n}{%
n-1+\Phi \left( n-1\right) }\left( d-f\left( a\right) \right) .  \label{3.3}
\end{equation}
\end{theorem}

\begin{proof}
\ The function $f$ \ is uniformly convex with $\Phi $ defined on $\left[
0,\infty \right) ,$ therefore from Remark \ref{Rem3} for $\lambda _{i}=\frac{%
1}{n},$ \ $i=1,...,n$%
\begin{equation}
\sum_{i=1}^{n}\frac{f\left( x_{i}\right) }{n}-f\left( \frac{%
\sum_{i=1}^{n}x_{i}}{n}\right) \geq \frac{1}{n}\sum_{i=1}^{n}\Phi \left(
\left\vert x_{i}-\frac{\sum_{i=1}^{n}x_{i}}{n}\right\vert \right)
\label{3.4}
\end{equation}%
holds.

In other words for $y_{i}=x_{i}-a,$ $i=1,...,n,$ 
\begin{equation}
\frac{1}{n}\sum_{i=1}^{n}\Phi \left( \left\vert y_{i}\right\vert \right)
\leq d-f\left( a\right)  \label{3.5}
\end{equation}%
holds.

From \ $\sum_{i=1}^{n}y_{i}=0$ \ we get that $\ \left\vert y_{n}\right\vert
=\left\vert -\sum_{i=1}^{n-1}y_{i}\right\vert .$ As $\Phi $ is positive,
according to Definition \ref{Def2}, $f$ is convex and also increasing, and
as it is given that $\Phi $\ is convex too and according to Remark \ref{Rem3}%
, $\Phi $ is increasing and $\Phi \left( 0\right) =0$, therefore, \ \ 
\begin{eqnarray}
\Phi \left( \left\vert y_{n}\right\vert \right) &=&\Phi \left( \left\vert
-\sum_{i=1}^{n-1}y_{i}\right\vert \right) \leq \Phi \left(
\sum_{i=1}^{n-1}\left\vert y_{i}\right\vert \right)  \label{3.6} \\
&=&\Phi \left( \left( n-1\right) \sum_{i=1}^{n-1}\frac{\Phi ^{-1}\left( \Phi
\left( \left\vert y_{i}\right\vert \right) \right) }{n-1}\right)  \notag \\
&\leq &\Phi \left( \left( n-1\right) \Phi ^{-1}\left( \frac{%
\sum_{i=1}^{n-1}\Phi \left( \left\vert y_{i}\right\vert \right) }{n-1}%
\right) \right) .  \notag
\end{eqnarray}%
Indeed, the left side inequality is because $\Phi $ is increasing and the
right side inequality follows because $\Phi ^{-1}$ is concave and $\Phi $ $\ 
$is increasing.

As $\Phi $ satisfies also $\Phi \left( AB\right) \leq \Phi \left( A\right)
\Phi \left( B\right) $ we get that 
\begin{equation}
\Phi \left( \left( n-1\right) \Phi ^{-1}\left( \frac{\sum_{i=1}^{n-1}\Phi
\left( \left\vert y_{i}\right\vert \right) }{n-1}\right) \right) \leq \frac{%
\Phi \left( n-1\right) }{n-1}\sum_{i=1}^{n-1}\Phi \left( \left\vert
y_{i}\right\vert \right)  \label{3.7}
\end{equation}%
and from (\ref{3.5}), (\ref{3.6}) and (\ref{3.7}) we obtain 
\begin{equation*}
\Phi \left( \left\vert y_{n}\right\vert \right) \leq \frac{\Phi \left(
n-1\right) }{n-1}\left( \sum_{i=1}^{n-1}\Phi \left( \left\vert
y_{i}\right\vert \right) -\Phi \left( \left\vert y_{n}\right\vert \right)
\right) .
\end{equation*}%
From the last inequality as $\Phi (x)$ is positive and\ increasing, together
with (\ref{3.5}) 
\begin{equation*}
\frac{n-1+\Phi \left( n-1\right) }{n-1}\Phi \left( \left\vert
y_{n}\right\vert \right) \leq \frac{\Phi \left( n-1\right) }{n-1}%
\sum_{i=1}^{n}\Phi \left( \left\vert y_{i}\right\vert \right) \leq \frac{%
\Phi \left( n-1\right) }{n-1}n\left( d-f\left( a\right) \right)
\end{equation*}%
holds, which is equivalent to (\ref{1.10}).
\end{proof}

\section{\textbf{Examples: Relations between superquadracity and some other
types of convex functions}}

In addition to the examples which appear in remarks \ref{Rem1}, \ref{Rem2}
and \ref{Rem5} and in examples \ref{Ex1} and \ref{Ex2}, the following
examples emphasize the relations between superquadracity and other
extensions of convexity.

\begin{example}
\label{Ex3} Let $f\left( x\right) =x^{2}\ln x$, $x\geq 0$. This function is
superquadratic (see \cite{AJS}) and negative on $\left[ 0,1\right] $. Hence
Definition \ref{Def3} is satisfied, and \ $f$ on $I=\left[ a,a+1\right] ,$ $%
a\geq 0$, is $\phi $-convex, where $\phi =-f$ is defined on $\left[ 0,1%
\right] $.

Moreover, $f$ according to Theorem \ref{Th2} possesses the $-\Gamma $
property as all superquadratic functions do, therefore $\phi $ possesses the
property $\Gamma $\ and as $\phi \left( 0\right) =0$\ then $f\in
E_{0}^{\Gamma }\left( I\right) $.

Also, according to Theorem \ref{Th3}, $\sqrt{\phi \left( x\right) }=x\sqrt{%
-\ln x}$ is subadditive. As this function is increasing on $\left( 0,e^{-1}%
\right] ,$ we get according to Lemma \ref{Lem2} that the function $-x\sqrt{%
-\ln x}$\ is superquadratic on $\left[ 0,e^{-1}\right] $ and according to
the definition of $\phi $-convexity, it is also $\phi $-convex\ where $\phi
=x\sqrt{-\ln x}$ on $\left[ 0,e^{-1}\right] $.
\end{example}

\begin{example}
\label{Ex4} From Definition \ref{Def1} and Definition \ref{Def3} it is
obvious that when a superquadratic function $f$ is negative, the function $f$
is $\phi $-convex where $\phi =-f$.

Therefore, according to \cite[Example 4.2]{AJS} the set of functions 
\begin{equation*}
f_{p}\left( x\right) =-\left( 1+x^{p}\right) ^{\frac{1}{p}},\quad p>0,
\end{equation*}%
are superquadratic and negative. Hence, these functions are also $\phi $%
-convex where $\phi =-f$.

The same holds for the function $f\left( x\right) =-x^{p}$, $0\leq p\leq 2$, 
$x\geq 0$ are $\phi $-convex where $\phi =-f\left( x\right) $.\ 
\end{example}

We get easily the following results:

\begin{example}
\bigskip \label{Ex5} Let $f\left( x\right) =-x^{p}$, $x\geq 1$, $1<p\leq 2$.
Then $f$ is $\phi $-convex where $\phi \left( x\right) =x^{q}$, $x\in \left[
0,1\right] $ and $1<q\leq p\leq 2$.
\end{example}

\begin{example}
\label{Ex6} The function 
\begin{equation*}
f\left( x\right) =\frac{1}{2}\ln \left( 1+x^{2}\right) -x\arctan \left(
x\right) ,\quad f\left( 0\right) =0
\end{equation*}

satisfies 
\begin{equation*}
f^{^{\prime }}\left( x\right) =-\arctan \left( x\right) \quad f^{^{\prime
}}\left( 0\right) =0
\end{equation*}%
and is superquadratic because $\left( \frac{f^{^{\prime }}\left( x\right) }{x%
}\right) ^{^{\prime }}>0$ (see \cite{AJS}).

Therefore, this function is superquadratic and because it is negative on $%
\left[ 0,\infty \right) $ it is also $\phi $-convex for $\phi =-f$.
\end{example}

\begin{example}
\label{Ex7} The function 
\begin{eqnarray*}
f\left( x\right) &=&\int_{0}^{x}\frac{t\left( t-2\right) }{\sqrt{t^{2}+1}}dt
\\
&=&\frac{1}{2}x\sqrt{x^{2}+1}-2\sqrt{x^{2}+1}-\frac{1}{2}\ln \left( x+\sqrt{%
x^{2}+1}\right) +2,\quad x\geq 0
\end{eqnarray*}%
is superquadratic, (see \cite{ABM}) and $f\left( x\right) \leq 0,$ $x\in %
\left[ 0,T\right] ,$ where $T$ satisfies $2<T<3$\ and $f\left( x\right) \geq
0$, $x\geq T$. Therefore, on $\left[ a,a+1\right] $, $a\geq 0$ is $\phi $%
-convex where $\phi =-f,$ $0\leq x\leq T$.
\end{example}

Another set of functions $f$ which are either $\phi $-convex or strongly
convex are:

\begin{example}
\label{Ex8} Let $f\left( x\right) =x\varphi \left( x\right) ,$ $x\geq 0$,
where\ $\varphi $ is convex. It is easy to verify that 
\begin{equation*}
tf\ \left( x\right) +\left( 1-t\right) f\left( y\right) \geq f\left(
tx+\left( 1-t\right) y\right) +t\left( 1-t\right) \varphi ^{^{\prime
}}\left( 0\right) \left( x-y\right) ^{2}
\end{equation*}%
holds. Therefore if $\varphi ^{^{\prime }}\left( 0\right) <0$\ \ it means
that $f:%
\mathbb{R}
_{+}\rightarrow 
\mathbb{R}
$ is $\phi $-convex, $\phi =-\varphi ^{^{\prime }}\left( 0\right) x^{2}$, $%
x\geq 0$.

If $\varphi ^{^{\prime }}\left( 0\right) >0$, it means that $f$ is strongly
convex and uniformly convex where $\Phi =\varphi ^{^{\prime }}\left(
0\right) x^{2}$ as explained in Example \ref{Ex2}.

In particular, the function $f\left( x\right) =x\left( x-1\right)
^{2n}=x\varphi \left( x\right) ,$ $x\geq 0$\ where the function $\varphi
\left( x\right) =\left( x-1\right) ^{2n}$, $n=1,2,...$\ is convex and
satisfies $\varphi ^{^{\prime }}\left( 0\right) =-2n$. Therefore, $f$ is $%
\phi $-convex where $\phi =2nx^{2}$.

In the special case that the function $f\left( x\right) =x\left( x-1\right)
^{2n+1}=x\varphi \left( x\right) ,$ $x\geq 0$ where the function $\varphi
\left( x\right) =\left( x-1\right) ^{2n+1}$, $n=1,2,...$\ is convex and
satisfies $\varphi ^{^{\prime }}\left( 0\right) =2n+1$. Therefore, $f$ is
strongly convex where $\phi =\left( 2n+1\right) x^{2}$.
\end{example}

\bigskip

\end{document}